\documentclass[11pt]{amsart}
\usepackage[usenames]{color}
\usepackage{fullpage}
\usepackage{amssymb, latexsym}
\usepackage[all,2cell,ps]{xy}
\usepackage{graphicx}

\theoremstyle{plain}
\newtheorem{thm}{Theorem}[section]
\newtheorem{theorem}[thm]{Theorem}

\newtheorem{lemma}[thm]{Lemma}

\newtheorem{proposition}[thm]{Proposition}

\theoremstyle{definition}
\newtheorem{de}[thm]{Definition}

\newtheorem{example}[thm]{Example}

\numberwithin{equation}{section}

\begin{document}

\title{Knot-theoretic flocks}

\author{Maciej Niebrzydowski}
\author{Agata Pilitowska}
\author{Anna Zamojska-Dzienio}

\address{(M.N.) Institute of Mathematics,
Faculty of Mathematics, Physics and Informatics,
University of Gda{\'n}sk, 80-308 Gda{\'n}sk, Poland}
\address{(A.P., A.Z.) Faculty of Mathematics and Information Science, Warsaw University of Technology, Koszykowa 75, 00-662 Warsaw, Poland}

\email{(M.N.) mniebrz@gmail.com}
\email{(A.P.) apili@mini.pw.edu.pl}
\email{(A.Z.) A.Zamojska-Dzienio@mini.pw.edu.pl}



\date{\today}

\begin{abstract}
We characterize the para-associative ternary quasigroups (flocks) applicable to knot theory, and show which of these structures are isomorphic. We enumerate them up to order 64.  We note that the operation used in knot-theoretic flocks has its non-associative version in extra loops. We use a group action on the set of flock colorings to improve the cocycle invariant associated with the knot-theoretic flock (co)homology. 
\end{abstract}

\maketitle
\section{Introduction and preliminary definitions}
Knot-theoretic ternary quasigroups are algebraic structures suitable for colorings of regions in the knot diagrams. Their operations generalize the ternary relations of the form $d=ab^{-1}c$ from the Dehn presentation of the knot group, just like the quandle operations generalize the conjugation present in the Wirtinger relations. Knot-theoretic ternary quasigroups are introduced in full generality in \cite{Nie17}, but see also \cite{Nie14} and \cite{NeNe17}. 

In this paper, we work with the sub-family of knot-theoretic ternary quasigroups introduced in \cite{Nie14}. They do not require orientation to produce coloring invariants of knots and knotted surfaces. We generalize the results of \cite{NPZ19}, where we described the structure of knot-theoretic ternary groups. Replacing the ternary associativity
condition with a similar condition of para-associativity, somewhat surprisingly leads to the structures based on nonabelian groups in place of the abelian ones.
The idempotent case corresponds to homomorphisms from the knot group, and the non-idempotent case involves a central involution in a group. It is possible to define various group actions on the set of flock colorings, which are compatible with the Reidemeister moves. We use one of them to strenghten the flock cocycle invariant obtained from the (co)homology of ternary algebras introduced in \cite{Nie17}. We also note the non-associative version of the obtained flock operations using extra loops. Let us begin with the necessary definitions.

A {\it ternary groupoid} is a non-empty set $X$ equipped with a ternary operation $[\, ]\colon X^3\to X$. It is denoted by $(X,[\, ])$.

A ternary groupoid $(X,[\, ])$ is called a \emph{ternary quasigroup} if for every $a,b,c\in X$ each of the following equations is uniquely solvable for $z\in X$:
\begin{align}
&[zab]=c, \\
&[azb]=c, \label{eq:m}\\
&[abz]=c.
\end{align}

We say that an operation $[\, ] \colon X^3 \to X$ is \emph{associative} if for all $a,b,c,d,e\in X$
\[
[[abc]de]=[a[bcd]e]=[ab[cde]].
\]

An associative ternary quasigroup is called a \emph{ternary group}; see \cite{Post} for a treatise on $n$-ary groups.

An operation $[\, ] \colon X^3 \to X$ is \emph{para-associative} if for all $a,b,c,d,e\in X$
\[
[[abc]de]=[a[dcb]e]=[ab[cde]].
\]

Various categories of para-associative groupoids were studied by Wagner in \cite{Wag53}. See also \cite{HoLa17}.

A para-associative ternary quasigroup is called a \emph{flock}. The connection of flocks with affine geometry was investigated by Dudek in \cite{Dud99}.

We say that a ternary groupoid $(X,[\, ])$ is \emph{idempotent} if $[aaa]=a$ for all $a\in X$. Idempotent flocks were used, for example, in \cite{Cer43}.

\begin{figure}
\begin{center}
\includegraphics[height=3 cm]{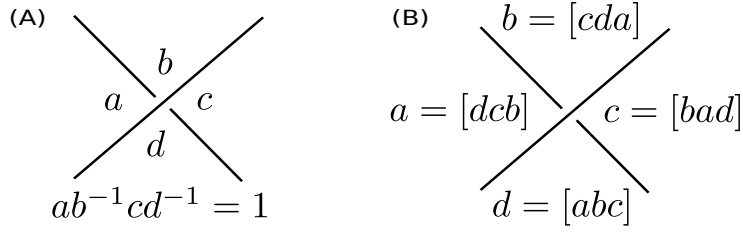}
\caption{A relation in the Dehn presentation can be realized using a para-associative operation $[xyz]=xy^{-1}z$.}\label{Dehngen}
\end{center}
\end{figure}

A good topological motivation for considering para-associativity comes from relations in the Dehn presentation of the knot group. Recall that in the Dehn presentation generators are assigned to the regions in the complement of a knot diagram $D$ on a plane, and relations correspond to the crossings and are as in Fig. \ref{Dehngen}(A). One of the generators, for example the one corresponding to the unbounded region, is set equal to identity. Geometrically, a generator can be viewed as a loop originating from a fixed point $P$ beneath the diagram, piercing a region to which it is assigned, and returning to $P$ through a region labeled with the identity element. See e.g. \cite{Kau83} for more details about Dehn presentation. 
Note that the fundamental group relations can be realized using a para-associative operation $[xyz]=xy^{-1}z$, see Fig. \ref{Dehngen}(B). We will show that para-associativity leaves a bit of room for generalizing this operation.

\begin{figure}
\begin{center}
\includegraphics[height=5 cm]{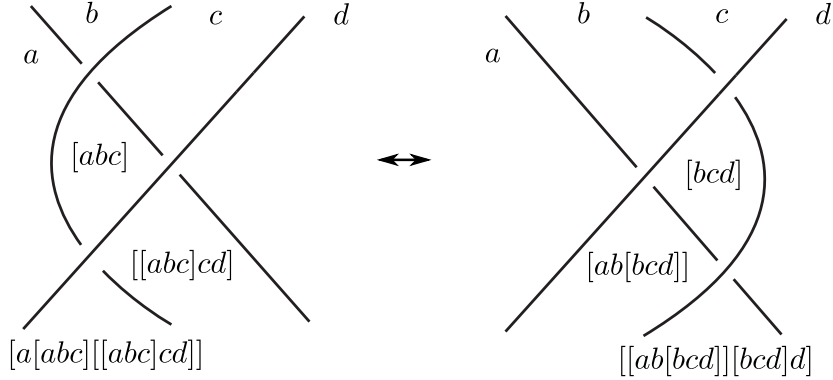}
\caption{The third Reidemeister move and the nesting conditions.}\label{rad3br}
\end{center}
\end{figure}

\begin{figure}
\begin{center}
\includegraphics[height=3 cm]{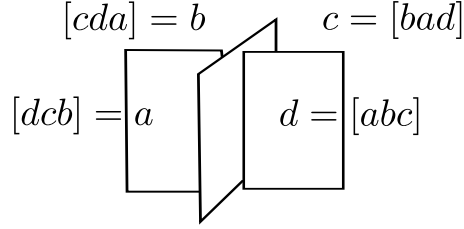}
\caption{Coloring of three-dimensional regions near a double point curve in a knotted surface diagram.}\label{surfdehn}
\end{center}
\end{figure}

The following two nesting conditions obtained from the coloring of regions in the third Reidemeister move were defined in \cite{Nie14}, see Fig. \ref{rad3br}:
\begin{equation}
\forall_{a,b,c,d\in X} \quad [ab[bcd]]=[a[abc][[abc]cd]],  \tag{LN}
\end{equation}
\begin{equation}
\forall_{a,b,c,d\in X} \quad [[abc]cd]=[[ab[bcd]][bcd]d]. \tag{RN} 
\end{equation}

By adding the adjective \emph{knot-theoretic} when talking about a ternary groupoid of some sort, we mean that the said groupoid satisfies the conditions LN and RN.
Thus we consider, for example, \emph{knot-theoretic ternary quasigroups},
\emph{knot-theoretic ternary groups}, and \emph{knot-theoretic flocks}.

Colorings of a knot or a knotted surface diagram $D$ with elements of a knot-theoretic flock $(X,[\, ])$ are defined in a simple way, as in Fig. \ref{Dehngen}(B) and Fig. \ref{surfdehn}. More specifically, they are functions $\mathcal{C}\colon Reg(D)\to X$, from the set of regions in the complement of $D$, such that a color $x$ near a crossing (resp. double-point curve) is expressed as $x=[yzw]$, where $y$, $z$, $w$ are the remaining three colors taken in a cyclic order, and the regions colored by $x$ and $y$ are separated by an over-arc (resp. over-sheet).
For Yoshikawa diagrams, it is required that near a marker the colors are assigned in an $a$, $b$, $a$, $b$ fashion, that is, opposite regions receive the same color.
The number of such colorings is an invariant under the applicable moves (Reidemeister, Roseman or Yoshikawa moves), see \cite{Nie17,KimNel} for more details.

In a ternary quasigroup $(X,[\, ])$, for every element $a\in X$, the unique solution of the equation $[aza]=a$  is called the \emph{skew element to} $a$ and is denoted by $\bar{a}$.

An operation $[\, ] \colon X^3 \to X$ is \emph{semi-commutative} if for all $a,b,c\in X$
\[
[abc]=[cba].
\]
It follows that semi-commutative flocks are ternary groups.

In \cite{NPZ19}, we obtained a precise characterization of knot-theoretic ternary groups, and applied it to the theory of flat links on possibly non-orientable surfaces. One of the main results of \cite{NPZ19} is as follows:

\begin{theorem}[{\cite[Corollary 4.5]{NPZ19}}]\label{mainchartgr}
Each knot-theoretic ternary group $(A,[\, ])$ is determined by an abelian group $(A,+)$ and an element $a\in A$ which is either zero (in the idempotent case) or has order two in $(A,+)$. 
For every $x,y,z\in A$
\[
[xyz]=x-y+z+a\quad {\rm and}\quad \bar{x}=x+a.
\]
\end{theorem}
In particular, we obtained a connection with Takasaki quandles (including dihedral quandles): 
\begin{equation*}\label{Takasaki}
[xy\bar{x}]=x-y+\bar{x}+a=x-y+x+a+a=2x-y.
\end{equation*}

We are now ready to consider the para-associative case.

\section{The structure of knot-theoretic flocks}

In a para-associative groupoid $(X,[\, ])$ an element $a\in X$ is called \emph{special} if
$[aax]=[xaa]$
for all $x\in X$.

In \cite{Dud99} the following theorem, analogous to the Gluskin-Hossz{\'u}
theorem (\cite{Hos63, DudGla, Glu65, Post}), was proven.

\begin{theorem}[{\cite[Proposition 4.5]{Dud99}}]\label{HGflock}
A para-associative groupoid $(X,[\, ])$ with a special element is a flock if and only if there exists a binary group $(X,\cdot)$ and its anti-automorphism $\theta$ such that $\theta^2(x)=x$, and
\begin{equation}
[xyz]=x\cdot\theta(y)\cdot z\cdot b
\end{equation}
for all $x$, $y$, $z\in X$, where $b$ is a central element of $(X,\cdot)$
such that $\theta(b)=b$.
\end{theorem}

First, we will show that knot-theoretic flocks satisfy the requirement of existence of a special element from the above theorem.

\begin{lemma}[{\cite[Section 3]{Dud99}}]
For any elements $x$, $y$, $z$ of a flock $(X,[\, ])$
\begin{align*}
&[y\bar{x}x]=[yx\bar{x}]=[x\bar{x}y]=[\bar{x}xy]=y,\\
&\bar{\bar{x}}=x,\\
&\overline{[xyz]}=[\bar{x}\,\bar{y}\,\bar{z}].
\end{align*}
\end{lemma}

\begin{lemma} \label{special}
Let $(X,[\, ])$ be a knot-theoretic flock. Then
every element $x\in X$ is special.
\end{lemma}
\begin{proof}
From para-associativity and the condition RN, for any $a$, $b$, $c$, $d\in X$ we have
\[
[d[bbc]a]=[[dcb]ba]=[[dc[cba]][cba]a]=[d[[cba][cba]c]a].
\]
From the uniqueness of solution of equation (\ref{eq:m}) it follows that
\[
[bbc]=[[cba][cba]c].
\]
Substituting $\bar{b}$ for $c$, we get
\[
b=[bb\bar{b}]=[[\bar{b}ba][\bar{b}ba]\bar{b}]=[aa\bar{b}],
\]
and after substituting $\bar{b}$ for $b$, we obtain $\bar{b}=[aab]$.
Similarly, using the condition LN and para-associativity:
\[
[a[cbb]d]=[ab[bcd]]=[a[abc][[abc]cd]]=[a[c[abc][abc]]d].
\]
It follows that
\[
[cbb]=[c[abc][abc]],
\]
and after substituting $\bar{b}$ for $c$, we get
\[
b=[\bar{b}bb]=[\bar{b}[ab\bar{b}][ab\bar{b}]]=[\bar{b}aa],
\]
which is equivalent to $\bar{b}=[baa]$.
To summarize, for any $a$, $b\in X$,
\begin{equation}
[aab]=\bar{b}=[baa], \label{skewfl}
\end{equation}
that is, every element is special.
\end{proof}

Note that if $\bar{x}=x$ for all $x\in X$, then such knot-theoretic flocks are examples of heaps (\cite{Wag53, HoLa17}), i.e. para-associative ternary groupoids $(X,[\, ])$ satisfying $[aax]=x=[xaa]$ for all $a$, $x\in X$.

\begin{proposition}
A flock $(X,[\,])$ in which the equation (\ref{skewfl}) is satisfied for any $x$, $y\in X$ is knot-theoretic.
\end{proposition}
\begin{proof}
Using $[yxx]=\bar{y}$ and para-associativity, we obtain the condition LN as follows:
\[
[ab[bcd]]=[a[cbb]d]=[a\bar{c}d]=[a[c[abc][abc]]d]=[a[abc][[abc]cd]].
\]

The equation $[xxy]=\bar{y}$, for $x$, $y\in X$, yields the condition RN:
\[
[[abc]cd]=[a[ccb]d]=[a\bar{b}d]=[a[[bcd][bcd]b]d]=[[ab[bcd]][bcd]d].
\]
\end{proof}

Now we obtain a characterization of knot-theoretic flocks.

\begin{theorem}\label{struct}
A ternary groupoid $(X,[\, ])$ is a knot-theoretic flock if and only if there exists a binary group $(X,\cdot)$ such that for any $x$, $y$, $z\in X$
\begin{equation}\label{knfl}
[xyz]=x\cdot y^{-1}\cdot z\cdot b,
\end{equation}
where $b$ is either the identity $e$ of the group $(X,\cdot)$, or is a central element of order two in $(X,\cdot)$. 
\end{theorem}
\begin{proof}
From Theorem \ref{HGflock} and Lemma \ref{special}
\[
[xyz]=x\cdot\theta(y)\cdot z\cdot b,
\]
where $\theta$ is an anti-automorphism such that $\theta(b)=b$, and $b$ is in the center of the group $(X,\cdot)$.
From (\ref{skewfl}) we have
\[
\bar{e}=[eee]=e\cdot \theta(e)\cdot e\cdot b=b.
\]
Then, for any $y\in X$,
\[
b=\bar{e}=[yye]=y\cdot \theta(y)\cdot e\cdot b.
\]
Thus, $e=y\cdot\theta(y)$, that is, $\theta(y)=y^{-1}$. Since $\theta(b)=b$, it follows that $b^2=e$. Note that, for any $x\in X$,
\[
\bar{x}=[xee]=x\cdot e^{-1}\cdot e\cdot b=x\cdot b.
\]
Now suppose that a ternary groupoid $(X,[\, ])$ has an operation of the form (\ref{knfl}). Then it is a ternary quasigroup and satisfies the para-associative condition:
\begin{align*}
&[[xyz]vw]=(x\cdot y^{-1}\cdot z\cdot b)\cdot v^{-1}\cdot w\cdot b=\\
&[x[vzy]w]=x\cdot(v\cdot z^{-1}\cdot y\cdot b)^{-1}\cdot w\cdot b=\\
&[xy[zvw]]=x\cdot y^{-1}\cdot (z\cdot v^{-1}\cdot w\cdot b)\cdot b,
\end{align*}
since $b$ is in the center of the group $(X,\cdot)$ and of order one or two.
The nesting conditions LN and RN are also easy to check. 
\end{proof}

Note that one can obtain the core group operation (see e.g. \cite{Wada} for details) in a knot-theoretic flock by taking
\begin{equation*}
[xy\bar{x}]=x\cdot y^{-1}\cdot \bar{x}\cdot b=x\cdot y^{-1}\cdot x\cdot b\cdot b
=x\cdot y^{-1}\cdot x.
\end{equation*} 

By Theorem \ref{struct} each knot-theoretic flock $(X,[\,])$ is defined by a group $(X,\cdot)$ and an element $b\in Z(X)$ of order one or two. In this case, we write $(X,[\,])=\mathcal{F}((X,\cdot),b)$ and call the group $(X,\cdot)$ \emph{associated} to the knot-theoretic flock $(X,[\,])$.
The next result shows the relationship between isomorphic knot-theoretic flocks and their associated groups and central elements.

\begin{theorem} \label{isomorphism}
Let $(X_1,[\, ]_1)=\mathcal{F}((X_1,\cdot),b_1)$ and $(X_2,[\, ]_2)=\mathcal{F}((X_2,\ast),b_2)$ be two knot-theoretic flocks.
Then knot-theoretic flocks $(X_1,[\, ]_1)$ and $(X_2,[\, ]_2)$ are isomorphic if and only if there exists a group isomorphism $f\colon (X_1,\cdot)\to (X_2,\ast)$ such that $b_2=f(b_1)$.
\end{theorem}

\begin{proof}
The implication ``$\Leftarrow$" directly follows by \cite[Proposition 4.10]{Dud99}.

To prove the converse let $h\colon (X_1,[\, ]_1)\to (X_2,[\, ]_2)$ be a flock isomorphism. This means that for $x,y\in X_1$ we have
\begin{align}\label{eq:ktfiso}
&h(x\cdot y)=h(x\cdot b_1^{-1}\cdot y \cdot b_1)= h([xb_1y]_1)=[h(x)h(b_1)h(y)]_2=h(x)\ast (h(b_1))^{-1}\ast h(y)\ast b_2.
\end{align}
(The first equality holds by centrality of the element $b_1$.)

Let us define the mapping
\[f\colon X_1\to X_2, \quad x\mapsto h(x)\ast (h(b_1))^{-1}\ast b_2.
\]
Clearly, $f(b_1)=h(b_1)\ast (h(b_1))^{-1}\ast b_2=b_2$. We will show that $f$ is a group isomorphism. Since $h\colon X_1\to X_2$ is a bijection, then $f$ is a bijection, too.  Further, by \eqref{eq:ktfiso} and centrality of the element $b_2\in X_2$ we immediately obtain for $x,y\in X_1$:
\begin{align*}
&f(x\cdot y)=h(x\cdot y)\ast (h(b_1))^{-1}\ast b_2\stackrel{ \eqref{eq:ktfiso}}=h(x)\ast (h(b_1))^{-1}\ast h(y)\ast b_2\ast (h(b_1))^{-1}\ast b_2=\\
&h(x)\ast (h(b_1))^{-1}\ast b_2\ast h(y)\ast (h(b_1))^{-1}\ast b_2=f(x)\ast f(y),
\end{align*}
which finishes the proof.
\end{proof}

Using Theorem \ref{isomorphism} and GAP \cite{GAP2019} (in particular, the Small Groups library), we were able to enumerate the knot-theoretic flocks obtained from non-abelian groups, up to order 64. Table \ref{count_all_nonab} lists the orders for which such flocks exist. For the ones obtained from abelian groups (that is, for knot-theoretic ternary groups), see \cite{NPZ19}. 

\begin{table}
$$\begin{array}{|r|cccccccccccccccccc|}\hline
n &6 &8 &10 &12 &14 &16 &18 &20 &21 &22 &24 &26 &27 &28 &30 &32 &34 &36\\\hline
\text{all}&1 &4 &1 &5 &1 &23 &3 &5 &1 &1 &24 &1 &2 &4 &3 &127 &1 &16\\
\text{idempotent}&1 &2 &1 &3 &1 &9 &3 &3 &1 &1 &12 &1 &2 &2 &3 &44 &1 &10\\\hline
\end{array}$$

$$\begin{array}{|r|cccccccccccccccccc|}\hline
n &38 &39 &40 &42 &44 &46 &48 &50 &52 &54 &55 &56 &57 &58 &60 &62 &63 &64\\\hline
\text{all} &1 &1 &23 &6 &4 &1 &112 &3 &5 &14 &1 &20 &1 &1 &17 &1 &2 &886\\
\text{idempotent} &1 &1 &11 &5 &2 &1 &47 &3 &3 &12 &1 &10 &1 &1 &11 &1 &2 &256\\\hline
\end{array}$$
\caption{The number of knot-theoretic flocks of size $n$, obtained from non-abelian groups, up to isomorphism.}
\label{count_all_nonab}
\end{table}

\subsection{A generalization to extra loops}

A {\it binary quasigroup} is a groupoid $(Q,*)$ such that the equation $x*y=z$ has a unique solution in $Q$ whenever two
of the three elements $x$, $y$, $z$ of $Q$ are specified. A {\it loop} $(L,*)$ is a quasigroup with an identity element $e$ such that $x*e=x=e*x$, for all $x\in L$. See, for example, \cite{Pfl90} for an introduction to the theory of loops and binary quasigroups.
An {\it extra loop} is a loop $(L,*)$ satisfying one of the following equivalent conditions:
\begin{align*}
&1.\ (x*(y*z))*y=(x*y)*(z*y),\\
&2.\ (y*z)*(y*x)=y*((z*y)*x),\\
&3.\ ((x*y)*z)*x=x*(y*(z*x)),
\end{align*}
for all $x$, $y$, and $z\in L$. Any group is an extra loop. A classical example of an extra loop that is not a group is as follows. Let $(G,\cdot)$ be a group, and $M(G,2)$ be the set 
$G\times\{0,1\}$ equipped with the operation
$(g,0)*(h,0)=(gh,0)$, $(g,0)*(h,1)=(hg,1)$, $(g,1)*(h,0)=(gh^{-1},1)$, and $(g,1)*(h,1)=(h^{-1}g,0)$. Then $(M(G,2),*)$ is a nonassociative Moufang loop if and only if $(G,\cdot)$ is nonabelian, and $(M(D_4,2),*)$, where $D_4$ is the dihedral group with eight elements, is an extra loop. The smallest nonassociative extra loops have 16 elements. The structure of extra loops was investigated, for example, in \cite{KiKu04}. In extra loops (and more generally in Moufang loops) the subloop generated by any two elements is 
a group. Elements have their inverses, satisfying the left and the right inverse properties: $x^{-1}*(x*y)=y$ and $(y*x)*x^{-1}=y$. 

We will show that the operation (\ref{knfl}) has its generalization in extra loops, but first we need a few more definitions.
For elements $x$, $y$, $z$ of a loop $L$, their {\it associator} $(x,y,z)\in L$ is defined as the unique element satisfying the equation
\[
(x*y)*z=(x,y,z)*(x*(y*z)).
\]
The {\it nucleus} $N(L)$ of $L$ consists of all elements $x\in L$ such that
\[
(x,y,z)=(y,x,z)=(y,z,x)=e,
\]
for all $y$, $z\in L$. The {\it center} $Z(L)$ is the subloop $\{ x\in N(L): x*y=y*x,\ \textrm{for all}\ y\in L\}$.

\begin{proposition}
Let $(L,*)$ be an extra loop. Then the operations $[xyz]_1=((x*y^{-1})*z)*k$ and 
$[xyz]_2=(x*(y^{-1}*z))*k$, where $k\in Z(L)$ is of order one or two, satisfy the conditions LN and RN. They can be used for defining knot/knotted surface coloring invariants via unoriented diagrams using the coloring scheme from Fig. \ref{Dehngen}(B) and Fig. \ref{surfdehn}. 
\end{proposition}
\begin{proof}
The operation $[\, ]_1$ generalizes the operation  $[xyz]=((x*y^{-1})*z)$ that was introduced in \cite{Nie14}. Note that since $k$ is in $Z(L)$, it can be moved throughout any word containing it. The operation $[\, ]_1$ is used an even number of times on the right and on the left hand sides of the axioms LN and RN, and therefore $k$ also appears an even number of times. Thus, since $k$ is of order at most two, it can be eradicated from these expressions, and the proof that $[\, ]_1$ satisfies LN and RN reduces to the proof contained in \cite[Lemma 5.7]{Nie14}. 

Note that because of our coloring conventions, the operation $[\, ]_1$ has to satisfy the equalities
\[ 
[[abc]_1cb]_1=a,\ [c[abc]_1a]_1=b,\ \textrm{and}\ [ba[abc]_1]_1=c,
\]
for all $a$, $b$, $c\in L$ (they are true for knot-theoretic flocks). Since $k$ appears twice on the left hand sides of these equations, it can be disregarded, and the equalities follow from the inverse properties of extra loops, as in \cite{Nie14}. The proofs for the operation $[xyz]_2$ are analogous.
\end{proof}

\section{Cocycle invariants and group actions on the set of colorings}

The (co)homology theory for algebras satisfying the nesting conditions LN and RN was developed in \cite{Nie17}. The operation (\ref{knfl}) defining knot-theoretic flocks considerably simplifies the terms appearing in \cite{Nie17}, and allows for an independent treatment. 

Let $(X,[\,])=\mathcal{F}((X,\cdot),k)$ be a knot theoretic flock. We will suppress the symbol of the group operation in the expressions below. The chain groups $C_n(X)=Z\langle X^{n+2}\rangle$ are defined as the free abelian groups generated by $(n+2)$-tuples $(x_0,x_1,\ldots, x_n,x_{n+1})$ of elements of $X$, for $n\geq -1$, with $C_{-2}(X)=\mathbb{Z}$. The differential $\partial_n\colon C_n(X) \to C_{n-1}(X)$ takes the form:
\begin{align*}
&\partial_n(x_0,x_1,\ldots,x_n,x_{n+1})=(x_1,\ldots,x_n,x_{n+1})\\
&+\sum_{i=1}^n(-1)^i\{(x_0x_i^{-1}x_{i+1}k^i,x_1x_i^{-1}x_{i+1}k^{i-1},\ldots,
x_{i-1}x_i^{-1}x_{i+1}k,\hat{x}_i,x_{i+1},\ldots,x_n,x_{n+1})\\
&+(x_0,x_1,\ldots,x_{i-1},\hat{x}_i,x_{i-1}x_i^{-1}x_{i+1}k,x_{i-1}x_i^{-1}x_{i+2}k^2,\ldots,x_{i-1}x_i^{-1}x_{n+1}k^{n+1-i})\}\\
&+(-1)^{n+1}(x_0,x_1,\ldots,x_n),
\end{align*}
where $\hat{x}_i$ denotes a missing element. We also set $\partial_{-1}(x_0)=0$ and
$\partial_0(x_0,x_1)=x_1-x_0$.
There is a degenerate subcomplex $\{C^D,\partial\}$ in which 
$C_n^D(X)$ is the free abelian group generated by $(n+2)$-tuples 
$x=(x_0,x_1,\ldots, x_n,x_{n+1})$ of elements of $X$ containing a triple
$a$, $b$, $ba^{-1}bk$ on three consecutive coordinates, for some $a$ and $b\in X$.
For $n<1$, we take $C_n^D(X)=0$. The normalized homology yields knot and knotted surface invariants via cycles assigned to colorings of oriented diagrams, see \cite{Nie17} and some details below. The knot-theoretic flock cohomology is defined in a standard dual way, with the coboundary $\delta$ obtained via $(\delta f)(c)=f(\partial c)$.
Thus, the 1-cocycles (used for link diagrams) are functions 
$f\colon X\times X\times X\to A$, where $A$ is an abelian group, satisfying
two conditions for all $a$, $b$, $c$, $d\in X$:
\begin{align*}
(1)\quad &f(a,b,ba^{-1}bk)=0,\\
(2)\quad &f(b,c,d)-f(a,ab^{-1}ck,ab^{-1}d)-f(ab^{-1}ck,c,d)\\
		&+f(a,b,bc^{-1}dk)+f(ac^{-1}d,bc^{-1}dk,d)-f(a,b,c)=0.
\end{align*}
2-cocycles (used for knotted surface diagrams in $\mathbb{R}^3$) are 
functions 
$\phi\colon X\times X\times X\times X\to A$, satisfying for all $a$, $b$, $c$, $d$, $e\in X$: 
\begin{align*}
(1)\quad &\phi(a,b,ba^{-1}bk,c)=\phi(c,a,b,ba^{-1}bk)=0,\\
(2)\quad &\phi(b,c,d,e)-\phi(a,ab^{-1}ck,ab^{-1}d,ab^{-1}ek)
-\phi(ab^{-1}ck,c,d,e)+\phi(a,b,bc^{-1}dk,bc^{-1}e)\\
&+\phi(ac^{-1}d,bc^{-1}dk,d,e)-\phi(a,b,c,cd^{-1}ek)
-\phi(ad^{-1}ek,bd^{-1}e,cd^{-1}ek,e)+\phi(a,b,c,d)=0.
\end{align*}
The cocycle invariants using knot-theoretic flocks are defined in analogy to the construction in \cite{CJKLS03}. 
A cocycle invariant can be viewed as a multiset consisting of evaluations of a given cocycle on the cycles assigned to all colorings of a given diagram. Here we give a more detailed explanation for links, directing the reader to \cite{Nie17} for a description of cocycle invariants for knotted surfaces. 

\begin{figure}
\begin{center}
\includegraphics[height=3.5 cm]{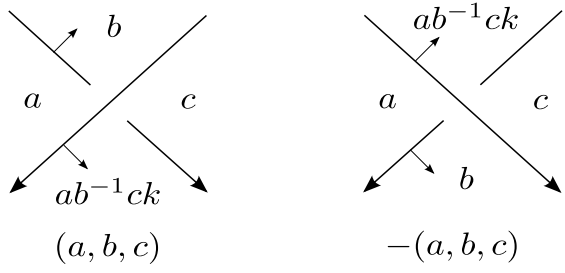}
\caption{The chains assigned to flock-colored crossings.}\label{flockcycles}
\end{center}
\end{figure}

Let $D$ be an oriented link diagram, and let $\mathcal{C}\colon Reg(D)\to X$ be its coloring with a knot-theoretic flock $(X,[\,])=\mathcal{F}((X,\cdot),k)$, where $Reg(D)$ denotes the set of regions in the complement of $D$. We denote the $\pm 1$ sign of the crossing $cr$ by $\epsilon(cr)$, its source region (i.e., the region with all the co-orientation arrows of $cr$ pointing out of it) by $r_s$, and the target region (all co-orientation arrows point into it) by $r_t$. The region of $cr$ separated from $r_s$ by an under-arc will be denoted by $r_m$. Then 
\[
c(\mathcal{C})=\sum_{cr\in D} \epsilon(cr)(\mathcal{C}(r_s),\mathcal{C}(r_m),\mathcal{C}(r_t))
\]
is a cycle in the first homology (with $\mathbb{Z}$ coefficients) of $(X,[\,])$. See Fig. \ref{flockcycles}, which shows the chains assigned to crossings. Now let $\{\mathcal{C}_1,\ldots,\mathcal{C}_n\}$ be the set of all colorings of $D$ with $(X,[\,])$, and let $\phi$ be a 1-cocycle from the cohomology of $(X,[\,])$, with values in the abelian group $A$. Then the cocycle invariant $\Psi(D,\phi)$ is defined as the multiset
\[
\{\phi(c(\mathcal{C}_1)),\ldots, \phi(c(\mathcal{C}_n))\}.
\]
It is an invariant of Reidemeister moves, so we can write $\Psi(L,\phi)$ instead of $\Psi(D,\phi)$, where $D$ is a diagram for a link $L$.

\begin{example}
Let $X$ be a permutation group with 12 elements numbered as follows:
1. (),\\
2. (1,2,3,5)(4,10,7,12)(6,11,9,8), 3. (1,3)(2,5)(4,7)(6,9)(8,11)(10,12), 
4. (1,4,8)(2,6,10)(3,7,11)(5,9,12),\\ 
5. (1,5,3,2)(4,12,7,10)(6,8,9,11), 
6. (1,6,3,9)(2,7,5,4)(8,10,11,12), 7. (1,7,8,3,4,11)(2,9,10,5,6,12),\\ 
8. (1,8,4)(2,10,6)(3,11,7)(5,12,9), 9. (1,9,3,6)(2,4,5,7)(8,12,11,10), 
10. (1,10,3,12)(2,11,5,8)(4,6,7,9),\\ 11. (1,11,4,3,8,7)(2,12,6,5,10,9), 
12. (1,12,3,10)(2,8,5,11)(4,9,7,6).\\
Note that $k=(1,3)(2,5)(4,7)(6,9)(8,11)(10,12)$ is of order two, and belongs to the center of $X$. From now on, we will write the numbers assigned to these elements instead of the elements themselves. The twelve tables below show an example of a 1-cocycle $\Phi$ for the knot theoretic flock $(X,[\,])=\mathcal{F}((X,\cdot),k)$ with values in $\mathbb{Z}_3$, generated with the help of GAP. The tables include the values of $\Phi$
on the triples $(i,j,l)$ with $i$, $j$, $l\in X$. The value of   
$\Phi(i,j,l)$ can be found in the $i$-th table, in the intersection of the $j$-th row and the $l$-th column. For example, $\Phi(1,2,3)=2$ and $\Phi(10,3,6)=1$. The cocycle invariant $\Psi(L,\Phi)$ is quite effective in distinguishing links. It has 27 distinct values on the 48-element set of nontrivial links with two components, that have up to eight crossings in the minimal braid form in the table from \cite{Git04}. It distinguishes all of them from the trivial link with two components, for which the value of the cocycle invariant is $12^3=1728$. The number of colorings can distinguish only 5 classes among these links. The values of the cocycle invariant in Table \ref{cocvalues} are given as polynomials. For example $768 + 408t + 552t^2$ calculated for the closure of the braid $\sigma_1\sigma_2\sigma_1^{-1}\sigma_2\sigma_1\sigma_3\sigma_2^{-1}\sigma_3$
means that out of 1728 cycles assigned to the colorings of this link, $\Phi$ has value 0 on 768 of them, value 1 on 408 cycles, and value 2 on 552 cycles. That is, the values of the cocycle are encoded in the powers of $t$.

\begin{tiny}
\noindent
\[ 
i=1\hspace{0.5cm}
\begin{array}{|cccccccccccc|} 
\hline
0 & 0 & 0 & 0 & 1 & 1 & 1 & 0 & 1 & 1 & 2 & 1 \\ 
0 & 2 & 2 & 0 & 0 & 2 & 0 & 1 & 0 & 2 & 0 & 1 \\ 
2 & 1 & 0 & 1 & 2 & 0 & 1 & 1 & 2 & 2 & 2 & 2 \\ 
0 & 1 & 2 & 2 & 2 & 0 & 2 & 2 & 0 & 1 & 0 & 1 \\ 
0 & 0 & 2 & 2 & 2 & 0 & 2 & 1 & 0 & 0 & 0 & 1 \\ 
0 & 2 & 1 & 1 & 2 & 0 & 1 & 2 & 0 & 1 & 1 & 2 \\ 
0 & 1 & 2 & 0 & 0 & 2 & 2 & 0 & 1 & 1 & 0 & 0 \\ 
1 & 2 & 0 & 0 & 0 & 1 & 0 & 2 & 0 & 1 & 1 & 0 \\ 
0 & 2 & 0 & 0 & 0 & 2 & 0 & 0 & 2 & 0 & 1 & 0 \\ 
0 & 2 & 1 & 2 & 0 & 1 & 1 & 0 & 1 & 1 & 0 & 0 \\ 
1 & 0 & 0 & 2 & 0 & 0 & 0 & 0 & 0 & 1 & 1 & 0 \\ 
0 & 0 & 0 & 1 & 1 & 1 & 2 & 2 & 0 & 2 & 1 & 2 \\ 
\hline
\end{array}
\hspace{0.7cm} i=2\hspace{0.5cm}
\begin{array}{|cccccccccccc|} 
\hline
1 & 0 & 2 & 0 & 0 & 0 & 2 & 1 & 0 & 2 & 0 & 2 \\ 
2 & 2 & 1 & 2 & 0 & 0 & 0 & 0 & 2 & 0 & 0 & 2 \\ 
2 & 0 & 1 & 0 & 1 & 0 & 1 & 0 & 1 & 1 & 0 & 2 \\ 
2 & 0 & 1 & 0 & 2 & 0 & 0 & 2 & 0 & 2 & 1 & 0 \\ 
0 & 0 & 2 & 1 & 0 & 0 & 0 & 2 & 0 & 2 & 0 & 2 \\
2 & 2 & 1 & 1 & 0 & 0 & 0 & 2 & 1 & 2 & 1 & 0 \\ 
2 & 0 & 0 & 2 & 2 & 0 & 0 & 1 & 0 & 0 & 1 & 0 \\ 
0 & 0 & 0 & 1 & 2 & 0 & 0 & 2 & 1 & 2 & 0 & 2 \\ 
2 & 1 & 2 & 2 & 1 & 0 & 2 & 2 & 0 & 0 & 1 & 0 \\ 
1 & 0 & 0 & 0 & 2 & 2 & 0 & 1 & 0 & 1 & 0 & 1 \\ 
2 & 0 & 1 & 0 & 2 & 0 & 2 & 1 & 1 & 2 & 0 & 0 \\ 
0 & 2 & 2 & 1 & 1 & 0 & 0 & 2 & 0 & 0 & 0 & 1 \\
\hline
\end{array}
\]
\[ 
i=3\hspace{0.5cm}
\begin{array}{|cccccccccccc|} 
\hline
0 & 0 & 1 & 2 & 0 & 0 & 0 & 0 & 0 & 1 & 1 & 2 \\ 
2 & 1 & 0 & 1 & 2 & 0 & 0 & 2 & 0 & 0 & 0 & 0 \\ 
0 & 1 & 1 & 2 & 1 & 0 & 0 & 0 & 0 & 0 & 0 & 0 \\ 
1 & 2 & 2 & 1 & 2 & 2 & 0 & 0 & 2 & 1 & 1 & 0 \\ 
2 & 0 & 0 & 1 & 0 & 0 & 0 & 2 & 2 & 2 & 0 & 0 \\ 
0 & 1 & 0 & 0 & 1 & 1 & 2 & 0 & 0 & 0 & 1 & 0 \\ 
2 & 0 & 0 & 1 & 1 & 2 & 0 & 0 & 2 & 2 & 1 & 0 \\ 
1 & 0 & 1 & 0 & 0 & 0 & 2 & 0 & 1 & 0 & 0 & 0 \\ 
1 & 1 & 0 & 1 & 0 & 1 & 0 & 1 & 0 & 0 & 1 & 2 \\ 
0 & 1 & 0 & 1 & 2 & 2 & 1 & 2 & 1 & 0 & 2 & 0 \\ 
0 & 2 & 2 & 0 & 1 & 0 & 0 & 0 & 2 & 0 & 2 & 2 \\ 
1 & 1 & 0 & 2 & 0 & 2 & 0 & 2 & 0 & 1 & 2 & 1 \\
\hline
\end{array}
\hspace{0.7cm} i=4\hspace{0.5cm}
\begin{array}{|cccccccccccc|} 
\hline
1 & 2 & 0 & 0 & 1 & 0 & 0 & 2 & 0 & 2 & 0 & 1 \\ 
1 & 2 & 2 & 0 & 2 & 0 & 1 & 2 & 0 & 0 & 1 & 2 \\ 
0 & 0 & 2 & 1 & 1 & 2 & 0 & 0 & 0 & 2 & 0 & 1 \\ 
0 & 0 & 0 & 0 & 2 & 0 & 0 & 0 & 0 & 0 & 0 & 0 \\ 
0 & 0 & 2 & 0 & 0 & 0 & 0 & 0 & 0 & 0 & 0 & 1 \\ 
0 & 1 & 2 & 0 & 0 & 2 & 2 & 1 & 1 & 2 & 1 & 0 \\ 
2 & 1 & 1 & 2 & 1 & 1 & 0 & 1 & 0 & 1 & 0 & 1 \\ 
2 & 2 & 0 & 0 & 0 & 0 & 2 & 2 & 1 & 2 & 0 & 1 \\ 
1 & 1 & 2 & 0 & 1 & 0 & 2 & 2 & 1 & 1 & 2 & 1 \\ 
2 & 0 & 1 & 0 & 0 & 1 & 0 & 0 & 0 & 2 & 0 & 1 \\ 
0 & 1 & 0 & 1 & 1 & 2 & 2 & 0 & 1 & 1 & 0 & 1 \\ 
2 & 0 & 0 & 0 & 1 & 2 & 1 & 1 & 0 & 0 & 2 & 0 \\ 
\hline
\end{array}
\]
\[ 
i=5\hspace{0.5cm}
\begin{array}{|cccccccccccc|} 
\hline
0 & 0 & 0 & 0 & 0 & 0 & 0 & 0 & 0 & 0 & 1 & 0 \\ 
0 & 0 & 0 & 1 & 1 & 1 & 0 & 1 & 0 & 2 & 2 & 0 \\ 
2 & 1 & 1 & 2 & 0 & 2 & 1 & 0 & 1 & 1 & 0 & 0 \\ 
2 & 2 & 0 & 2 & 0 & 1 & 2 & 1 & 0 & 2 & 0 & 2 \\ 
1 & 0 & 1 & 2 & 1 & 0 & 2 & 0 & 1 & 1 & 2 & 0 \\ 
2 & 2 & 0 & 2 & 0 & 2 & 2 & 2 & 2 & 0 & 0 & 2 \\ 
0 & 2 & 1 & 1 & 0 & 2 & 1 & 0 & 1 & 0 & 2 & 2 \\ 
1 & 0 & 2 & 2 & 0 & 1 & 0 & 2 & 0 & 1 & 2 & 1 \\ 
2 & 2 & 2 & 0 & 0 & 0 & 1 & 1 & 0 & 0 & 0 & 0 \\ 
0 & 0 & 0 & 1 & 1 & 0 & 1 & 2 & 2 & 1 & 2 & 0 \\ 
0 & 1 & 0 & 1 & 0 & 0 & 0 & 0 & 1 & 2 & 2 & 1 \\ 
0 & 0 & 0 & 2 & 1 & 0 & 0 & 1 & 0 & 2 & 0 & 0 \\
\hline
\end{array}
\hspace{0.7cm} i=6\hspace{0.5cm}
\begin{array}{|cccccccccccc|} 
\hline
1 & 2 & 1 & 0 & 1 & 0 & 0 & 2 & 1 & 2 & 2 & 0 \\ 
0 & 0 & 0 & 1 & 2 & 0 & 1 & 0 & 2 & 1 & 0 & 0 \\ 
2 & 2 & 0 & 0 & 0 & 0 & 2 & 0 & 0 & 1 & 2 & 1 \\ 
1 & 0 & 0 & 2 & 0 & 0 & 2 & 2 & 1 & 2 & 0 & 1 \\ 
0 & 2 & 0 & 2 & 0 & 0 & 2 & 0 & 0 & 1 & 0 & 0 \\ 
2 & 1 & 2 & 2 & 1 & 1 & 1 & 1 & 0 & 2 & 0 & 1 \\ 
0 & 2 & 2 & 2 & 0 & 0 & 2 & 2 & 0 & 0 & 2 & 2 \\ 
1 & 1 & 0 & 2 & 1 & 0 & 1 & 1 & 0 & 2 & 0 & 1 \\ 
0 & 0 & 0 & 2 & 0 & 1 & 1 & 0 & 0 & 1 & 1 & 2 \\ 
0 & 1 & 2 & 0 & 0 & 1 & 0 & 2 & 0 & 2 & 0 & 2 \\ 
1 & 2 & 2 & 0 & 1 & 0 & 1 & 0 & 1 & 2 & 2 & 1 \\ 
2 & 1 & 1 & 0 & 0 & 1 & 0 & 0 & 2 & 1 & 1 & 1 \\
\hline
\end{array}
\]
\[ 
i=7\hspace{0.5cm}
\begin{array}{|cccccccccccc|} 
\hline
0 & 2 & 2 & 0 & 0 & 2 & 0 & 0 & 1 & 1 & 2 & 1 \\ 
2 & 0 & 2 & 1 & 0 & 0 & 0 & 0 & 0 & 0 & 2 & 2 \\ 
2 & 2 & 0 & 1 & 2 & 0 & 1 & 0 & 0 & 2 & 0 & 1 \\ 
0 & 1 & 0 & 0 & 2 & 0 & 0 & 1 & 0 & 0 & 1 & 1 \\ 
2 & 0 & 0 & 0 & 1 & 1 & 0 & 2 & 2 & 2 & 0 & 1 \\ 
0 & 0 & 1 & 2 & 1 & 1 & 0 & 2 & 2 & 1 & 1 & 1 \\ 
1 & 1 & 1 & 0 & 0 & 1 & 2 & 1 & 1 & 1 & 1 & 1 \\ 
0 & 0 & 1 & 2 & 0 & 1 & 0 & 2 & 2 & 2 & 2 & 0 \\ 
0 & 2 & 0 & 2 & 1 & 0 & 0 & 2 & 0 & 1 & 1 & 0 \\ 
2 & 2 & 0 & 1 & 0 & 1 & 0 & 1 & 1 & 1 & 0 & 2 \\ 
0 & 0 & 1 & 1 & 2 & 1 & 1 & 2 & 1 & 2 & 2 & 1 \\ 
0 & 0 & 0 & 0 & 1 & 2 & 0 & 2 & 0 & 0 & 2 & 2 \\ 
\hline
\end{array}
\hspace{0.7cm} i=8\hspace{0.5cm}
\begin{array}{|cccccccccccc|} 
\hline
2 & 0 & 2 & 0 & 0 & 0 & 0 & 1 & 0 & 1 & 2 & 0 \\ 
2 & 0 & 0 & 0 & 1 & 0 & 1 & 0 & 0 & 0 & 2 & 0 \\ 
2 & 1 & 0 & 2 & 1 & 1 & 0 & 1 & 0 & 1 & 2 & 1 \\ 
2 & 1 & 0 & 0 & 2 & 1 & 0 & 0 & 1 & 0 & 0 & 0 \\ 
0 & 2 & 1 & 0 & 0 & 2 & 1 & 0 & 0 & 1 & 2 & 1 \\
2 & 2 & 2 & 0 & 1 & 2 & 0 & 0 & 2 & 2 & 2 & 1 \\ 
0 & 1 & 0 & 2 & 0 & 0 & 2 & 0 & 1 & 2 & 0 & 0 \\
1 & 1 & 1 & 1 & 0 & 0 & 0 & 0 & 2 & 1 & 0 & 1 \\
0 & 2 & 2 & 0 & 2 & 0 & 2 & 0 & 2 & 0 & 2 & 2 \\
2 & 1 & 2 & 0 & 1 & 1 & 1 & 0 & 0 & 2 & 0 & 0 \\
1 & 0 & 0 & 0 & 1 & 1 & 1 & 2 & 2 & 1 & 0 & 1 \\ 
0 & 2 & 2 & 0 & 1 & 2 & 1 & 0 & 0 & 2 & 1 & 0 \\
\hline
\end{array}
\]
\[ 
i=9\hspace{0.55cm}
\begin{array}{|cccccccccccc|} 
\hline
1 & 1 & 0 & 0 & 0 & 2 & 1 & 2 & 0 & 2 & 2 & 1 \\ 
1 & 2 & 0 & 0 & 0 & 1 & 1 & 1 & 0 & 0 & 1 & 0 \\ 
2 & 0 & 0 & 1 & 2 & 2 & 1 & 2 & 0 & 2 & 0 & 2 \\ 
0 & 1 & 2 & 1 & 0 & 1 & 0 & 0 & 0 & 2 & 2 & 0 \\ 
0 & 0 & 0 & 2 & 1 & 2 & 0 & 1 & 0 & 0 & 0 & 2 \\ 
2 & 1 & 1 & 2 & 2 & 0 & 1 & 2 & 0 & 2 & 0 & 2 \\ 
0 & 2 & 1 & 0 & 2 & 0 & 0 & 2 & 0 & 0 & 0 & 1 \\ 
2 & 0 & 2 & 0 & 0 & 2 & 0 & 1 & 0 & 2 & 1 & 2 \\ 
1 & 0 & 1 & 0 & 0 & 0 & 1 & 2 & 2 & 0 & 0 & 1 \\ 
2 & 0 & 2 & 2 & 1 & 1 & 1 & 2 & 1 & 1 & 0 & 2 \\ 
2 & 2 & 1 & 2 & 2 & 2 & 1 & 1 & 0 & 1 & 1 & 0 \\ 
0 & 0 & 0 & 0 & 1 & 0 & 0 & 0 & 1 & 0 & 0 & 1 \\
\hline
\end{array}
\hspace{0.7cm} i=10\hspace{0.35cm}
\begin{array}{|cccccccccccc|} 
\hline
1 & 1 & 0 & 2 & 2 & 2 & 0 & 1 & 1 & 0 & 0 & 1 \\ 
0 & 2 & 2 & 1 & 1 & 2 & 2 & 2 & 0 & 1 & 0 & 1 \\ 
0 & 0 & 0 & 2 & 0 & 1 & 2 & 0 & 2 & 0 & 1 & 0 \\ 
1 & 1 & 0 & 1 & 2 & 1 & 2 & 1 & 1 & 0 & 1 & 1 \\
0 & 0 & 1 & 0 & 1 & 0 & 0 & 2 & 0 & 0 & 0 & 0 \\ 
1 & 0 & 0 & 0 & 0 & 2 & 0 & 0 & 0 & 1 & 1 & 1 \\ 
2 & 0 & 2 & 2 & 2 & 0 & 0 & 2 & 1 & 0 & 0 & 0 \\ 
0 & 1 & 0 & 0 & 0 & 0 & 0 & 2 & 2 & 0 & 0 & 2 \\ 
0 & 2 & 0 & 2 & 0 & 0 & 0 & 0 & 1 & 0 & 1 & 2 \\ 
1 & 0 & 0 & 2 & 2 & 0 & 1 & 1 & 0 & 1 & 0 & 0 \\ 
1 & 0 & 2 & 1 & 0 & 2 & 0 & 0 & 2 & 0 & 2 & 2 \\ 
0 & 1 & 0 & 0 & 2 & 1 & 0 & 2 & 2 & 1 & 1 & 0 \\ 
\hline
\end{array}
\]
\[ 
i=11\hspace{0.45cm}
\begin{array}{|cccccccccccc|} 
\hline
0 & 2 & 0 & 0 & 0 & 1 & 2 & 0 & 1 & 2 & 2 & 2 \\ 
2 & 0 & 0 & 1 & 0 & 1 & 1 & 1 & 2 & 0 & 0 & 0 \\ 
0 & 0 & 0 & 0 & 0 & 0 & 0 & 0 & 0 & 2 & 1 & 2 \\ 
0 & 1 & 0 & 0 & 1 & 1 & 0 & 2 & 2 & 0 & 0 & 0 \\ 
0 & 2 & 2 & 1 & 2 & 1 & 1 & 0 & 2 & 1 & 0 & 2 \\ 
2 & 0 & 1 & 0 & 0 & 2 & 2 & 0 & 0 & 2 & 0 & 1 \\
0 & 1 & 2 & 1 & 2 & 2 & 0 & 2 & 0 & 0 & 2 & 2 \\
0 & 1 & 0 & 0 & 1 & 2 & 0 & 0 & 0 & 2 & 1 & 0 \\ 
0 & 2 & 1 & 2 & 2 & 1 & 0 & 1 & 0 & 1 & 0 & 1 \\ 
2 & 1 & 2 & 2 & 0 & 1 & 0 & 1 & 0 & 2 & 0 & 0 \\
0 & 2 & 0 & 2 & 0 & 0 & 0 & 0 & 0 & 0 & 1 & 0 \\
0 & 1 & 0 & 0 & 2 & 1 & 2 & 0 & 0 & 1 & 0 & 0 \\
\hline
\end{array}
\hspace{0.7cm} i=12\hspace{0.35cm}
\begin{array}{|cccccccccccc|} 
\hline
1 & 1 & 0 & 0 & 0 & 1 & 2 & 1 & 2 & 2 & 1 & 0 \\ 
2 & 0 & 1 & 2 & 2 & 0 & 1 & 1 & 2 & 2 & 0 & 1 \\ 
1 & 1 & 0 & 0 & 1 & 0 & 1 & 0 & 0 & 2 & 2 & 0 \\ 
0 & 2 & 1 & 2 & 1 & 1 & 0 & 0 & 0 & 0 & 0 & 0 \\ 
0 & 0 & 0 & 1 & 0 & 0 & 2 & 2 & 0 & 1 & 1 & 0 \\ 
0 & 0 & 0 & 0 & 0 & 1 & 2 & 0 & 0 & 0 & 0 & 1 \\ 
0 & 0 & 0 & 2 & 2 & 1 & 2 & 2 & 1 & 1 & 2 & 0 \\ 
0 & 2 & 0 & 0 & 2 & 0 & 2 & 0 & 1 & 1 & 1 & 0 \\ 
2 & 0 & 0 & 0 & 2 & 1 & 1 & 1 & 0 & 1 & 2 & 0 \\ 
2 & 0 & 1 & 1 & 0 & 2 & 2 & 0 & 0 & 0 & 0 & 2 \\ 
0 & 0 & 0 & 1 & 0 & 0 & 0 & 0 & 2 & 0 & 1 & 0 \\ 
2 & 0 & 1 & 1 & 1 & 2 & 1 & 0 & 0 & 0 & 0 & 0 \\
\hline
\end{array}
\]

\end{tiny}

\begin{table}
\begin{tabular}{|r|l|}
\hline
$480 + 264t + 120t^2$ & $\sigma_1\sigma_1,\ \sigma_1\sigma_1\sigma_1\sigma_1\sigma_2^{-1}\sigma_1\sigma_2^{-1},\ \sigma_1\sigma_1\sigma_2^{-1}\sigma_1\sigma_2^{-1}\sigma_3\sigma_2^{
-1}\sigma_3$\\
$480 + 408t + 552t^2$ & $\sigma_1\sigma_1\sigma_1\sigma_1,\ \sigma_1\sigma_2^{-1}\sigma_1\sigma_2^{-1}\sigma_3^{-1}\sigma_2\sigma_4\sigma_3^{-1}\sigma_4$\\
$768 + 120t + 552t^2$ & $\sigma_1\sigma_1\sigma_2\sigma_1^{-1}\sigma_2,\ \sigma_1\sigma_1\sigma_2^{-1}\sigma_1\sigma_1\sigma_2^{-1}\sigma_2^{-1},\sigma_1\sigma_1\sigma_1\sigma_1\sigma_
2\sigma_2\sigma_1^{-1}\sigma_2\sigma_2$,\\ & $\sigma_1\sigma_1\sigma_1\sigma_2\sigma_1^{-1}\sigma_2\sigma_3\sigma_2^{-1}\sigma_3\sigma_3$\\
$864 + 576t$ & $\sigma_1\sigma_1\sigma_2^{-1}\sigma_1\sigma_2^{-1}$\\
$1152$ & $\sigma_1\sigma_1\sigma_1\sigma_1\sigma_1\sigma_1,\ \sigma_1\sigma_1\sigma_2\sigma_1^{-1}\sigma_2\sigma_3\sigma_2^{-1}\sigma_3$\\
$912 + 408t + 408t^2$ & $\sigma_1\sigma_2^{-1}\sigma_1\sigma_3\sigma_2^{-1}\sigma_3,\ \sigma_1\sigma_1\sigma_1\sigma_2\sigma_1^{-1}\sigma_2\sigma_2$\\
$864$ & $\sigma_1\sigma_1\sigma_1\sigma_1\sigma_2\sigma_1^{-1}\sigma_2,\ \sigma_1\sigma_1\sigma_1\sigma_2^{-1}\sigma_1\sigma_3\sigma_2^{-1}\sigma_3,\ \sigma_1\sigma_1\sigma_1\sigma_
2\sigma_1^{-1}\sigma_2\sigma_2\sigma_3\sigma_2^{-1}\sigma_3$,\\ & $\sigma_1\sigma_1\sigma_1\sigma_1\sigma_1\sigma_2\sigma_1^{-1}\sigma_2\sigma_2,\ \sigma_1\sigma_1\sigma_2\sigma_1^{-1}\sigma_
3^{-1}\sigma_2\sigma_4\sigma_3^{-1}\sigma_4$\\
$480 + 120t + 264t^2$ & $\sigma_1\sigma_1\sigma_2^{-1}\sigma_1\sigma_2^{-1}\sigma_3^{-1}\sigma_2\sigma_3^{-1}$\\
$1152 + 288t$ & $\sigma_1\sigma_1\sigma_1\sigma_2^{-1}\sigma_1\sigma_1\sigma_2^{-1}$\\
$624 + 408t + 120t^2$ & $\sigma_1\sigma_1\sigma_1\sigma_2^{-1}\sigma_1\sigma_2^{-1}\sigma_2^{-1},\ \sigma_1\sigma_1\sigma_2\sigma_1^{-1}\sigma_2\sigma_3^{-1}\sigma_2\sigma_3^{
-1},\ \sigma_1\sigma_1\sigma_2^{-1}\sigma_1\sigma_3\sigma_2^{-1}\sigma_3\sigma_3$,\\ & $\sigma_1\sigma_2^{-1}\sigma_1\sigma_2^{-1}\sigma_2^{-1}\sigma_3\sigma_2^{-1}\sigma_3$\\
$624 + 264t + 552t^2$ & $\sigma_1\sigma_1\sigma_2^{-1}\sigma_1\sigma_3\sigma_2\sigma_2\sigma_3,\ \sigma_1\sigma_2^{-1}\sigma_1\sigma_3\sigma_2^{-1}\sigma_2^{-1}\sigma_2^{-1}\sigma_3$\\
$1440 + 288t$ & $\sigma_1\sigma_1\sigma_2^{-1}\sigma_1\sigma_2^{-1}\sigma_1\sigma_2^{-1}$\\
$864 + 576t$ & $\sigma_1\sigma_1\sigma_1\sigma_2\sigma_1^{-1}\sigma_3^{-1}\sigma_2\sigma_3^{-1}$\\
$768 + 408t + 264t^2$ & $\sigma_1\sigma_1\sigma_1\sigma_2^{-1}\sigma_1^{-1}\sigma_1^{-1}\sigma_2^{-1}$\\
$768 + 264t + 408t^2$ & $\sigma_1\sigma_1\sigma_1\sigma_2\sigma_1\sigma_1\sigma_2,\ \sigma_1\sigma_1\sigma_2^{-1}\sigma_1\sigma_1\sigma_3\sigma_2^{-1}\sigma_3,\ \sigma_1\sigma_1\sigma_
1\sigma_1\sigma_1\sigma_1\sigma_2\sigma_1^{-1}\sigma_2,$\\ & $\sigma_1\sigma_1\sigma_1\sigma_1\sigma_2\sigma_1^{-1}\sigma_2\sigma_3\sigma_2^{-1}\sigma_3$\\
$1152 + 288t$ & $\sigma_1\sigma_1\sigma_1\sigma_2\sigma_1^{-1}\sigma_1^{-1}\sigma_2,\ \sigma_1\sigma_2^{-1}\sigma_1\sigma_2^{-1}\sigma_3\sigma_2^{-1}\sigma_2^{-1}\sigma_3$\\
$480 + 552t + 408t^2$ & $\sigma_1\sigma_1\sigma_1\sigma_1\sigma_1\sigma_1\sigma_1\sigma_1$\\
$768 + 552t + 120t^2$ & $\sigma_1\sigma_1\sigma_2\sigma_1^{-1}\sigma_2\sigma_3\sigma_2^{-1}\sigma_3\sigma_4\sigma_3^{-1}\sigma_4$\\
$624 + 408t + 408t^2$ & $\sigma_1\sigma_2^{-1}\sigma_1\sigma_3\sigma_2\sigma_2\sigma_4^{-1}\sigma_3\sigma_4^{-1}$\\
$1440$ & $\sigma_1\sigma_1\sigma_2^{-1}\sigma_1\sigma_3\sigma_2^{-1}\sigma_2^{-1}\sigma_3$\\
$864 + 288t + 288t^2$ & $\sigma_1\sigma_2^{-1}\sigma_1\sigma_2^{-1}\sigma_1\sigma_3\sigma_2^{-1}\sigma_3,\ \sigma_1\sigma_1\sigma_2^{-1}\sigma_1^{-1}\sigma_1^{-1}\sigma_3\sigma_2^{
-1}\sigma_3$\\
$1632 + 1128t + 696t^2$ & $\sigma_1\sigma_2^{-1}\sigma_3\sigma_2^{-1}\sigma_1\sigma_2^{-1}\sigma_3\sigma_2^{-1}$\\
$1776 + 552t + 1128t^2$ & $\sigma_1\sigma_1\sigma_2\sigma_3^{-1}\sigma_2\sigma_1^{-1}\sigma_2\sigma_3\sigma_3\sigma_2$\\
$624 + 696t + 408t^2$ & $\sigma_1\sigma_1\sigma_1\sigma_1\sigma_2\sigma_1^{-1}\sigma_2\sigma_2\sigma_2$\\
$912 + 696t + 120t^2$ & $\sigma_1\sigma_1\sigma_2\sigma_1^{-1}\sigma_2\sigma_2\sigma_2\sigma_3\sigma_2^{-1}\sigma_3$\\
$912 + 264t + 552t^2$ & $\sigma_1\sigma_1\sigma_2\sigma_1\sigma_1\sigma_3^{-1}\sigma_2\sigma_3^{-1}$\\
$768 + 408t + 552t^2$ & $\sigma_1\sigma_2\sigma_1^{-1}\sigma_2\sigma_1\sigma_3\sigma_2^{-1}\sigma_3$\\
\hline
\end{tabular}
\caption{The values of the cocycle $\Phi$ on the 48 links with two components in the minimal braid form with up to 8 crossings.}
\label{cocvalues}
\end{table}
\end{example}

Let groups $G$ and $G'$ act on sets $A$ and $B$, respectively. Then the actions are said to be {\it equivalent} if there is an isomorphism $\beta\colon G\to G'$ and a bijective map $\alpha\colon A\to B$ such that, for all $g\in G$ and $a\in A$,
\begin{equation}\label{equivalentact}
\alpha(a{\mathchar"5E}g)=\alpha(a){\mathchar"5E}\beta(g),
\end{equation}
where ${\mathchar"5E}$ denotes the appropriate group action.
This defines an equivalence relation on group actions; see e.g. \cite{Rose} for more material on this topic. We will now consider group actions on the set of colorings, but for this notion to be useful, it has to take into account the Reidemeister (or Roseman, etc.) moves. We write the following definition on a more general level of knot-theoretic ternary quasigroups (see \cite{Nie17} for the corresponding colorings, generalizing the flock colorings).
\begin{de}\label{compatible}
Let $\mathfrak{C}(D,(X,[\, ]))$ denote the set of colorings of a knot diagram $D$ with a knot-theoretic ternary quasigroup $(X,[\, ])$. A Reidemeister move changing $D$ into a diagram $D'$ results in local changes of colorings, yielding a bijection 
$\alpha\colon\mathfrak{C}(D,(X,[\, ]))\to\mathfrak{C}(D',(X,[\, ]))$. Suppose that there is a group action 
$G\times\mathfrak{C}(D,(X,[\, ]))\to\mathfrak{C}(D,(X,[\, ]))$. We say that this action is {\it compatible with the Reidemeister moves}, if for any such move, and $D'$ and $\alpha$ as above, there is an isomorphism $\beta\colon G\to G'$, with $G'$ acting on 
$\mathfrak{C}(D',(X,[\, ]))$, so that the actions, $\alpha$, and $\beta$, satisfy the equation (\ref{equivalentact}). 
\end{de}

\begin{figure}
\begin{center}
\includegraphics[height=4.5 cm]{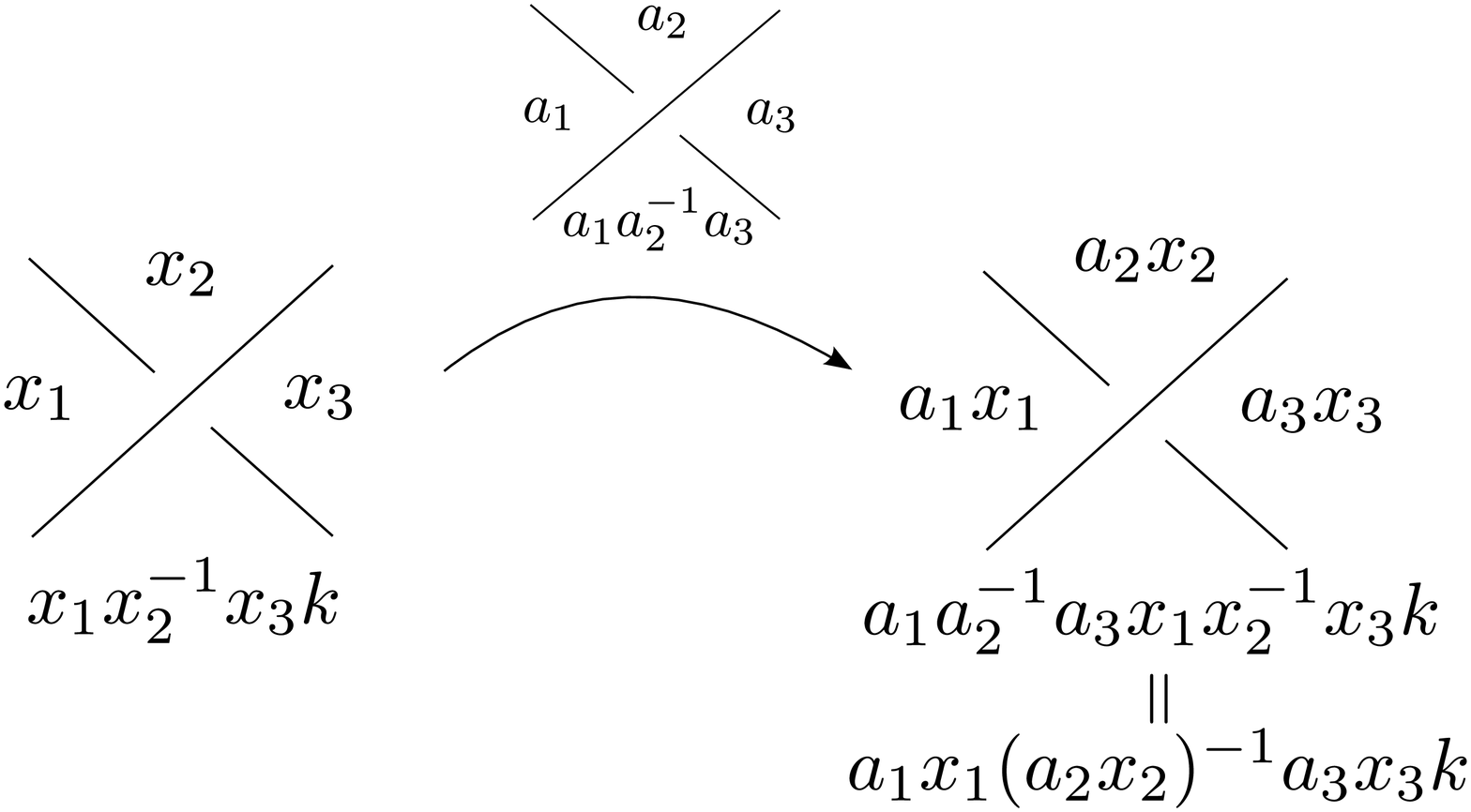}
\caption{The group of central colorings acting on colorings.}\label{action1}
\end{center}
\end{figure}

\begin{lemma}\label{directmult}
Let $\mathfrak{C}(D,\mathcal{F}((Z(X),\cdot),e))$ be the set of flock colorings of a diagram $D$ using elements of the center $Z(X)$ of the group $X$, with $e$ being the identity element of the group $X$. Then it forms a group acting on the set of colorings $\mathfrak{C}(D,\mathcal{F}((X,\cdot),k))$, for any central involution $k$ of $X$. The action is compatible with the Reidemeister moves.
\end{lemma}
\begin{proof}
The colorings from $\mathfrak{C}(D,\mathcal{F}((Z(X),\cdot),e))$ form a group with the operation of region-wise multiplication of colors. The identity element is the coloring in which all the regions are labeled by $e$. The inverse of a coloring $\mathcal{C}\in \mathfrak{C}(D,\mathcal{F}((Z(X),\cdot),e))$ is the coloring $\mathcal{C}^{-1}$ such that $\mathcal{C}^{-1}(r)=(\mathcal{C}(r))^{-1}$, for any region $r$ of $D$. We also note that taking a product of two colorings with elements from the center, but with $k\neq e$, gives a coloring with $k=e$. If $D$ and $D'$ differ by a Reidemeister move, then the corresponding groups of central colorings are isomorphic. 
There is an action of the group $\mathfrak{C}(D,\mathcal{F}((Z(X),\cdot),e))$ on the set of colorings $\mathfrak{C}(D,\mathcal{F}((X,\cdot),k))$ given by the region-wise multiplication of colors. Centrality of colors of the acting coloring ensures that the result is again a flock coloring, for any central involution $k\in X$ (see Fig. \ref{action1}). It is not difficult to see that this action is compatible with the Reidemeister moves: if $\alpha$ is as in Definition \ref{compatible}, then we can take $\beta$ to be the restriction of $\alpha$ to $\mathfrak{C}(D,\mathcal{F}((Z(X),\cdot),e))$.
\end{proof}

\begin{figure}
\begin{center}
\includegraphics[height=3.5 cm]{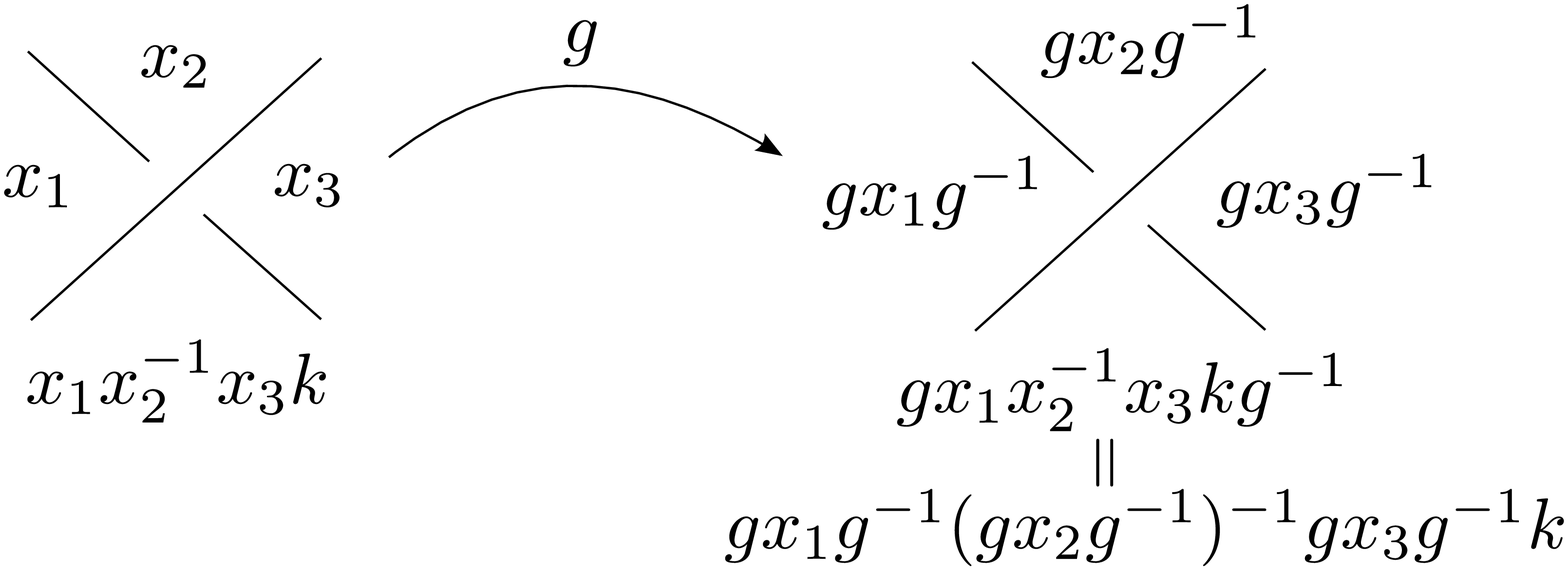}
\caption{A group $X$ acting on colorings by conjugation.}\label{action2}
\end{center}
\end{figure}

\begin{example}\label{actionbyconjugation}
Another example of an action of a group on the set of colorings that is compatible with the Reidemeister moves is given by the group $X$ (and its subgroups) acting on $\mathfrak{C}(D,\mathcal{F}((X,\cdot),k))$ by conjugation. More specifically, there is a mapping 
$g\times\mathcal{C}\mapsto\mathcal{C}{\mathchar"5E}g$ defined by 
$\mathcal{C}{\mathchar"5E}g(r)=g\mathcal{C}(r)g^{-1}$ for all regions $r$ of the diagram $D$. Centrality of $k$ is important, as can be seen in Fig. \ref{action2}.
Definition \ref{compatible} is satisfied, as we can take $\beta$ to be the identity isomorphism. Then $\alpha(\mathcal{C}{\mathchar"5E}g)=(\alpha(\mathcal{C})){\mathchar"5E}g$.
\end{example}

\begin{example}
We can generalize the action from Example \ref{actionbyconjugation}. Let $H$ and $S$ be subgroups of $X$. Then the direct product $H\times S$ acts on $\mathfrak{C}(D,\mathcal{F}((X,\cdot),k))$ via $\mathcal{C}{\mathchar"5E}(h,s)(r)=h\mathcal{C}(r)s^{-1}$, for
$(h,s)\in H\times S$.
\end{example}

Now we incorporate the actions on colorings compatible with the Reidemeister moves into cocycle invariants.

\begin{lemma}\label{cocyinvrefi}
Let $(X, [\, ])$ be a knot-theoretic ternary quasigroup, and $D$ be an oriented link diagram.
Suppose that $G\times \mathfrak{C}(D,(X,[\, ]))\to\mathfrak{C}(D,(X,[\, ]))$ is a group action compatible with the Reidemeister moves, and that 
\[
\mathcal{O}_1=\{\mathcal{C}_1,\ldots,\mathcal{C}_{n_1}\},
\mathcal{O}_2=\{\mathcal{C}_{n_1+1},\ldots,\mathcal{C}_{n_2}\},\ldots,
\mathcal{O}_s=\{\mathcal{C}_{n_{s-1}+1},\ldots,\mathcal{C}_{n_s}\}
\]
are the orbits of this action.
Let $\phi$ be a cocycle from the first cohomology of $(X,[\, ])$ with values in an abelian group $A$. Then the multiset of multisets
\[
\{\{\phi(c(\mathcal{C}_1)),\ldots,\phi(c(\mathcal{C}_{n_1}))\},
\{\phi(c(\mathcal{C}_{n_1+1})),\ldots,\phi(c(\mathcal{C}_{n_2}))\},\ldots,
\{\phi(c(\mathcal{C}_{n_{s-1}+1})),\ldots,\phi(c(\mathcal{C}_{n_s}))\}\}
\]
is a refinement of the cocycle invariant, that is not changed by the Reidemeister moves.
\end{lemma}
\begin{proof}
Let $\alpha\colon\mathfrak{C}(D,(X,[\, ]))\to\mathfrak{C}(D',(X,[\, ]))$, $\beta$ and $G'$ be as in Definition \ref{compatible}.
We have: $\mathcal{O}_i=\{\mathcal{C}_{n_{i-1}+1},\ldots,\mathcal{C}_{n_i}\}$ is an orbit of the action of $G$ if and only if 
$\alpha(\mathcal{O}_i)=\{\alpha(\mathcal{C}_{n_{i-1}+1}),\ldots,\alpha(\mathcal{C}_{n_i})\}$ is an orbit of the action of $G'$. Indeed:
\[
\mathcal{C}_i{\mathchar"5E}g=\mathcal{C}_j \iff 
\alpha(\mathcal{C}_i{\mathchar"5E}g)=\alpha(\mathcal{C}_j) \iff
\alpha(\mathcal{C}_i){\mathchar"5E}\beta(g)=\alpha(\mathcal{C}_j).
\]
In \cite{Nie17} we proved that the homology class of a cycle assigned to a knot-theoretic ternary quasigroup coloring is not changed by the Reidemeister moves. It follows that 
\[
c(\mathcal{C}{\mathchar"5E}g)\sim c(\alpha(\mathcal{C}{\mathchar"5E}g))
,\ \textrm{and}\ \phi(c(\mathcal{C}{\mathchar"5E}g))=\phi(c(\alpha(\mathcal{C}{\mathchar"5E}g))),
\]
for all $\mathcal{C}\in\mathfrak{C}(D,(X,[\, ]))$ and $g\in G$,
what ends the proof.
\end{proof}

\begin{example}
We can use Lemma \ref{cocyinvrefi}, and the action from Example \ref{actionbyconjugation}, to
improve the cocycle invariant $\Psi(L,\Phi)$ on the set of links from Table \ref{cocvalues}. More specifically, we consider the action by conjugation with the three-element subgroup generated by $(1,4,8)(2,6,10)(3,7,11)(5,9,12)$. For the closures of the braids
$\sigma_1\sigma_1\sigma_1\sigma_2\sigma_1\sigma_1\sigma_2$,
$\sigma_1\sigma_1\sigma_2^{-1}\sigma_1\sigma_1\sigma_3\sigma_2^{-1}\sigma_3$,
$\sigma_1\sigma_1\sigma_1\sigma_1\sigma_1\sigma_1\sigma_2\sigma_1^{-1}\sigma_2$,
and $\sigma_1\sigma_1\sigma_1\sigma_1\sigma_2\sigma_1^{-1}\sigma_2\sigma_3\sigma_2^{-1}\sigma_3$,
the value of the cocycle invariant $\Psi(L,\Phi)$ is $768 + 264t + 408t^2$. In particular, there are 1440 colorings for each of these links. Also, in each case the set of colorings 
splits into 216 one-element orbits and 408 three-element orbits. This also divides the sets of cycles corresponding to the colorings. The values of the cocycle $\Phi$ on these groupings of cycles allow us to separate the four braids into two classes. We will write the results as polynomials in the brackets with multiplicities, with a multiplicity giving the number of orbits with the same value of the polynomial (which describes the values of $\Phi$ on a given subset of cycles). 
For the first two braids we have:
\[
\{ 132 [ 1 ], 212 [ 3 ], 60 [ t ], 68 [3t],
 24 [ t^2 ], 128 [3t^2] \},
\]
where, for example, $128 [3t^2]$ means that there are 128 three-element orbits such that 
$\Phi$ has value 2 on each coloring in the orbit. For the third and the fourth braid the results are:
\[
\{ 132 [ 1 ], 212 [ 3 ], 24 [t], 80 [3t], 
60 [t^2], 116 [3t^2] \}.
\]
Thus, some additional links are distinguished.
\end{example}

\end{document}